\documentclass[12pt,leqno,draft]{article}
\usepackage{amsfonts}
\usepackage{amsmath, amsthm, amsfonts, amssymb, color}
\usepackage{mathrsfs}
\usepackage{color}
\newtheorem{thm}{Theorem}[section]

\theoremstyle{definition}
\newtheorem{defn}{Definition}[section]
\newcommand{\scr}[1]{\mathscr #1}
\definecolor{wco}{rgb}{0.5,0.2,0.3}

\numberwithin{equation}{section} \theoremstyle{remark}

\newcommand{\ua}{\uparrow}

\title{{\bf  Regularities and Exponential Ergodicity in Entropy  for    SDEs Driven by  Distribution Dependent Noise}\footnote{ Supported in
 part by   the National Key R\&D Program of China (No. 2022YFA1006000, 2020YFA0712900) and   NNSFC (12271398, 11921001). }}
\author{
{\bf   Xing Huang,  Feng-Yu Wang   }\\
\footnotesize{ Center for Applied Mathematics, Tianjin
University, Tianjin 300072, China}\\
\footnotesize{  xinghuang@tju.edu.cn},\ \ 
\footnotesize{  wangfy@tju.edu.cn}
}
\begin{document}
\allowdisplaybreaks
\def\R{\mathbb R}  \def\ff{\frac} \def\ss{\sqrt} \def\B{\mathbf
B}
\def\N{\mathbb N} \def\kk{\kappa} \def\m{{\bf m}}
\def\ee{\varepsilon}\def\ddd{D^*}
\def\dd{\delta} \def\DD{\Delta} \def\vv{\varepsilon} \def\rr{\rho}
\def\<{\langle} \def\>{\rangle} \def\GG{\Gamma} \def\gg{\gamma}
  \def\nn{\nabla} \def\pp{\partial} \def\E{\mathbb E}
\def\d{\text{\rm{d}}} \def\bb{\beta} \def\aa{\alpha} \def\D{\scr D}
  \def\si{\sigma} \def\ess{\text{\rm{ess}}}
\def\beg{\begin} \def\beq{\begin{equation}}  \def\F{\scr F}
\def\Ric{\text{\rm{Ric}}} \def\Hess{\text{\rm{Hess}}}
\def\e{\text{\rm{e}}} \def\ua{\underline a} \def\OO{\Omega}  \def\oo{\omega}
 \def\tt{\tilde} \def\Ric{\text{\rm{Ric}}}
\def\cut{\text{\rm{cut}}} \def\P{\mathbb P} \def\ifn{I_n(f^{\bigotimes n})}
\def\C{\scr C}   \def\G{\scr G}   \def\aaa{\mathbf{r}}     \def\r{r}
\def\gap{\text{\rm{gap}}} \def\prr{\pi_{{\bf m},\varrho}}  \def\r{\mathbf r}
\def\Z{\mathbb Z} \def\vrr{\varrho} \def\ll{\lambda}
\def\L{\scr L}\def\Tt{\tt} \def\TT{\tt}\def\II{\mathbb I}
\def\i{{\rm in}}\def\Sect{{\rm Sect}}  \def\H{\mathbb H}
\def\M{\scr M}\def\Q{\mathbb Q} \def\texto{\text{o}} \def\LL{\Lambda}
\def\Rank{{\rm Rank}} \def\B{\scr B} \def\i{{\rm i}} \def\HR{\hat{\R}^d}
\def\to{\rightarrow}\def\l{\ell}\def\iint{\int}
\def\EE{\scr E}\def\no{\nonumber}
\def\A{\scr A}\def\V{\mathbb V}\def\osc{{\rm osc}}
\def\BB{\scr B}\def\Ent{{\rm Ent}}\def\3{\triangle}\def\H{\scr H}
\def\U{\scr U}\def\8{\infty}\def\1{\lesssim}\def\HH{\mathrm{H}}
 \def\T{\scr T}
 \def\R{\mathbb R}  \def\ff{\frac} \def\ss{\sqrt} \def\B{\mathbf
B} \def\W{\mathbb W}
\def\N{\mathbb N} \def\kk{\kappa} \def\m{{\bf m}}
\def\ee{\varepsilon}\def\ddd{D^*}
\def\dd{\delta} \def\DD{\Delta} \def\vv{\varepsilon} \def\rr{\rho}
\def\<{\langle} \def\>{\rangle} \def\GG{\Gamma} \def\gg{\gamma}
  \def\nn{\nabla} \def\pp{\partial} \def\E{\mathbb E}
\def\d{\text{\rm{d}}} \def\bb{\beta} \def\aa{\alpha} \def\D{\scr D}
  \def\si{\sigma} \def\ess{\text{\rm{ess}}}
\def\beg{\begin} \def\beq{\begin{equation}}  \def\F{\scr F}
\def\Ric{\text{\rm{Ric}}} \def\Hess{\text{\rm{Hess}}}
\def\e{\text{\rm{e}}} \def\ua{\underline a} \def\OO{\Omega}  \def\oo{\omega}
 \def\tt{\tilde} \def\Ric{\text{\rm{Ric}}}
\def\cut{\text{\rm{cut}}} \def\P{\mathbb P} \def\ifn{I_n(f^{\bigotimes n})}
\def\C{\scr C}      \def\aaa{\mathbf{r}}     \def\r{r}
\def\gap{\text{\rm{gap}}} \def\prr{\pi_{{\bf m},\varrho}}  \def\r{\mathbf r}
\def\Z{\mathbb Z} \def\vrr{\varrho} \def\ll{\lambda}
\def\L{\scr L}\def\Tt{\tt} \def\TT{\tt}\def\II{\mathbb I}
\def\i{{\rm in}}\def\Sect{{\rm Sect}}  \def\H{\mathbb H}
\def\M{\scr M}\def\Q{\mathbb Q} \def\texto{\text{o}} \def\LL{\Lambda}
\def\Rank{{\rm Rank}} \def\B{\scr B} \def\i{{\rm i}} \def\HR{\hat{\R}^d}
\def\to{\rightarrow}\def\l{\ell}\def\iint{\int}
\def\EE{\scr E}\def\Cut{{\rm Cut}}
\def\A{\scr A} \def\Lip{{\rm Lip}}
\def\BB{\scr B}\def\Ent{{\rm Ent}}\def\L{\scr L}
\def\R{\mathbb R}  \def\ff{\frac} \def\ss{\sqrt} \def\B{\mathbf
B}
\def\N{\mathbb N} \def\kk{\kappa} \def\m{{\bf m}}
\def\dd{\delta} \def\DD{\Delta} \def\vv{\varepsilon} \def\rr{\rho}
\def\<{\langle} \def\>{\rangle} \def\GG{\Gamma} \def\gg{\gamma}
  \def\nn{\nabla} \def\pp{\partial} \def\E{\mathbb E}
\def\d{\text{\rm{d}}} \def\bb{\beta} \def\aa{\alpha} \def\D{\scr D}
  \def\si{\sigma} \def\ess{\text{\rm{ess}}}
\def\beg{\begin} \def\beq{\begin{equation}}  \def\F{\scr F}
\def\Ric{\text{\rm{Ric}}} \def\Hess{\text{\rm{Hess}}}
\def\e{\text{\rm{e}}} \def\ua{\underline a} \def\OO{\Omega}  \def\oo{\omega}
 \def\tt{\tilde} \def\Ric{\text{\rm{Ric}}}
\def\cut{\text{\rm{cut}}} \def\P{\mathbb P} \def\ifn{I_n(f^{\bigotimes n})}
\def\C{\scr C}      \def\aaa{\mathbf{r}}     \def\r{r}
\def\gap{\text{\rm{gap}}} \def\prr{\pi_{{\bf m},\varrho}}  \def\r{\mathbf r}
\def\Z{\mathbb Z} \def\vrr{\varrho} \def\ll{\lambda}
\def\L{\scr L}\def\Tt{\tt} \def\TT{\tt}\def\II{\mathbb I}
\def\i{{\rm in}}\def\Sect{{\rm Sect}}  \def\H{\mathbb H}
\def\M{\scr M}\def\Q{\mathbb Q} \def\texto{\text{o}} \def\LL{\Lambda}
\def\Rank{{\rm Rank}} \def\B{\scr B} \def\i{{\rm i}} \def\HR{\hat{\R}^d}
\def\to{\rightarrow}\def\l{\ell}
\def\8{\infty}\def\I{1}\def\U{\scr U}\def\beq{\begin{equation}}
\maketitle

\begin{abstract}  As two crucial tools  characterizing    regularity properties  of stochastic systems,  the  log-Harnack inequality and Bismut formula  have been intensively studied for distribution dependent (McKean-Vlasov) SDEs. However, due to technical difficulties, existing results mainly focus on the case with distribution free noise.
In this paper,  we introduce  a noise decomposition argument to  establish the    log-Harnack inequality and Bismut formula for
  SDEs with distribution dependent noise, in both non-degenerate and degenerate situations.  As application, the exponential ergodicity in entropy  is investigated.

\end{abstract} \noindent
 \vskip 2cm

\section{Introduction}
Let $\scr P(\R^d)$ be the space  of all  probability measures on $\R^d$ equipped with the weak topology.
Consider the following distribution dependent SDE on $\R^d$:
\beq\label{E1} \d X_t= b_t(X_t, \L_{X_t})\d t+  \si_t(X_t, \L_{X_t})\d B_t,\ \ t\in [0,T], \end{equation}
where  $T>0$ is a fixed time, $\L_{X_t}$ is the distribution of $X_t$,
$$b: [0,T]\times\R^d\times\scr P(\R^d)\to\R^d,\ \ \si: [0,T]\times\R^d\times\scr P(\R^d)  \to \R^d\otimes\R^d$$
are measurable,  and $B_t$ is a  $d$-dimensional Brownian motion on a complete filtration probability space $(\Omega,\F,\{\F_t\}_{t\in[0,T]},\P)$.

We investigate the regularity in initial distributions for  solutions to \eqref{E1}. More precisely, for $k> 1$ let
$$\scr P_k(\R^d):=\big\{\mu\in \scr P(\R^d):\ \|\mu\|_k:= \mu(|\cdot|^k)^{\ff 1 k}<\infty\big\},$$
which is a Polish space under the $L^k$-Wasserstein distance
$$\W_k(\mu,\nu)= \inf_{\pi\in \C(\mu,\nu)} \bigg(\int_{\R^{d}\times\R^{d}} |x-y|^k \pi(\d x,\d y)\bigg)^{\ff 1 {k}},\ \  \mu,\nu\in \scr P_k(\R^d), $$ where $\C(\mu,\nu)$ is the set of all couplings of $\mu$ and $\nu$.
When \eqref{E1} is well-posed for distributions in $\scr P_k(\R^d)$, i.e. for any $\F_0$-measurable initial value $X_0$  with $\L_{X_0}\in \scr P_k(\R^d)$ (correspondingly, any initial distribution $\mu\in \scr P_k(\R^d)$), the SDE \eqref{E1} has a unique solution (correspondingly, a unique weak solution) with $\L_{X_\cdot}\in C([0,T], \scr P_k(\R^d)),$  we consider the regularity of the maps
$$\scr P_k(\R^d)\ni \mu\mapsto P_t^*\mu:=\L_{X_t}\ \text{for}\ \L_{X_0}=\mu,\ \ t\in (0,T].$$
Since $P_t^*\mu$ is uniquely determined by
\beq\label{PT}  P_t f(\mu):=\int_{\R^d} f \d(P_t^*\mu),\ \ f\in \B_b(\R^d),\end{equation}
where $\B_b(\R^d)$ is the space of bounded measurable functions on $\R^d$, we study the regularity of functionals
$$\scr P_k(\R^d) \ni \mu \mapsto P_tf(\mu), \ \ \ t\in (0,T], f\in \B_b(\R^d).$$

When the noise is distribution free, i.e. $ \si_t(x, \mu)=\si_t(x)$ does not depend on the distribution argument $\mu$, the log-Harnack inequality
\beq\label{LH} P_t\log f(\mu)\le \log P_tf(\nu)+ \ff{c}t \W_2(\mu,\nu)^2,\ \ 0<f\in \B_b(\R^d), t\in (0,T], \mu,\nu\in \scr P_2(\R^d),\end{equation}
for some constant $c>0$ has been established in  \cite{HW18, HW22a, RW, FYW1}  under different conditions,   see also \cite{HRW19,HS} for extensions to the infinite-dimensional case.   A crucial application of this inequality  is that it is equivalent to the entropy-cost estimate
\begin{equation*} \Ent(P_t^*\nu|P_t^*\mu)\le \ff{c}t \W_2(\mu,\nu)^2,\ \ t\in (0,T], \mu,\nu\in \scr P_2(\R^d),\end{equation*}
where $\Ent(\nu|\mu)$ is the relative entropy of $\nu$ with respect to $\mu$.
With this estimate, the exponential ergodicity of $P_t^*$ in entropy is proved in \cite{RW} for a class of time-homogeneous distribution dependent SDEs.  The study of
\eqref{LH} goes back to \cite{W97,W10} where the family of  dimension-free Harnack inequalities is introduced, see \cite{Wbook} for various applications of this type inequalities.
We  emphasize that  arguments used   in the above mentioned references do not apply to distribution dependent noise.  The only known  log-Harnack inequality for distribution dependent noise
is established in \cite{BH} for  Ornstein-Ulenbeck type SDEs whose solutions are Gaussian processes and thus   easy to manage.

Another crucial tool characterizing the regularity of $\mu\mapsto P_t^*\mu$ is the following Bismut type formula for the intrinsic derivative $D^I$ in $\mu\in\scr P_k(\R^d)$ (see Definition 2.1 below):
\beq\label{BS} \beg{split}&D^I_\phi P_tf(\mu)= \E\bigg[ f(X_t^\mu) \int_0^t\<M_s^{\mu,\phi},\d B_s\>\bigg],\\
&\ \ \ \qquad t\in (0,T],\ f\in \B_b(\R^d), \phi\in L^{k}(\R^d\to\R^d;\mu),\end{split} \end{equation}
where   $\int_0^t\<M_s^{\mu,\phi},\d B_s\>$  is a martingale depending on $\mu$ and $\phi$ and $X_t^\mu$ solves \eqref{E1} from initial distribution $\mu$.

 Bismut formula was first established in \cite{Bim} for the  derivative formula of diffusion semigroups on Riemannian manifolds by using Malliavin calculus, which is also   called Bismut-Elworthy-Li formula due to \cite{EL}  where the martingale method is developed.
When $\si_t(x,\mu)=\si_t(x)$ is distribution free, this type formulas have been established in
\cite{BRW,BBP, HW22a,RW19,FYW3} under different conditions.

In the distribution dependent setting, Bismut  formula  is studied in    \cite{CM} for  the   decoupled SDEs  with fixed distribution parameter,  while \eqref{BS}    is derived in   \cite{B} for the Dirac measure  $\mu=\delta_x, x\in\R^d$.  An   implicit Bismut formula is  presented in \cite{Song,T} where the noise is allowed to be distribution dependent. So far, an explicit Bismut formula is still open for distribution dependent noise.   Nevertheless,   intrinsic  derivative estimates have been presented  for a class of SDEs with distribution dependent noise, see    \cite{HW21} and references therein.  This convinces us   of establishing the log-Harnack inequality and explicit Bismut formula  for  SDEs with distribution dependent noise.

 In this paper, we propose a noise decomposition argument which reduces the study of distribution dependent noise to  distribution free noise.  For simplicity, we only explain here the idea on establishing  the
 log-Harnack inequality for    the following distribution dependent SDE:
 \beq\label{EMV} \d X_t= b_t(X_t, \L_{X_t})\d t+  \si_t( \L_{X_t})\d B_t,\ \ t\in [0,T]. \end{equation}
Assume that $\si_t$ is bounded and Lipschitz continuous on $\scr P_2(\R^d)$, such that
 $$ (\si_t\si_t^*)(\gg)\ge  2\ll^2 I_d,\ \ \gg\in \scr P_2(\R^d) $$ holds  for some constant  $\ll>0$,
 where $I_d$ is the $d\times d$ identity matrix.  We take
 $$\tt\si_t(\gg):= \ss{(\si_t\si_t^*)(\gg)- \ll^2 I_d}.$$
Then $\tt\si_t(\gg)\ge \ll I_d$, and    \cite[Lemma 3.3]{PW} implies that $\tt\si_t(\gg)$ is Lipschitz continuous in $\gg\in\scr P_{2}(\R^d)$ as well.
Moreover,  for two independent $d$-dimensional Brownian motions $W_t$ and $\tt W_t$,
 $$\d B_t:= \si_t(\L_{X_t})^{-1} \big\{\ll \d W_t+ \tt \si_t(\L_{X_t})\d \tt W_t\big\}$$
 is a $d$-dimensional Brownian motion, so that \eqref{EMV} is reduced to
  \beq\label{E'} \d X_t= b_t(X_t, \L_{X_t})\d t+  \ll \d W_t+ \tt \si_t( \L_{X_t})\d \tt W_t,\ \ t\in [0,T]. \end{equation}
 Thus, by the well-posedness, \eqref{EMV} and \eqref{E'} provide the same operator $P_t$.
Now, consider the conditional probability $\P^{\tt W}$ given $\tt W$, under which
 $\int_0^t \tt \si_s( \L_{X_s})\d \tt W_s$ is deterministic so that  \eqref{E'} becomes an  SDE
 with constant noise $\ll \d W_t$, and hence its  log-Harnack inequality follows from  existing arguments developed for  distribution free noise.

  However, this noise decomposition argument is hard to extend to spatial-distribution dependent noise. So, in the following   we  only  consider
  \eqref{EMV} or \eqref{E'}, rather than \eqref{E1}.

Closely related to the log-Harnack inequality,   a very nice  entropy estimate has been derived in  \cite{BRS} for two SDEs with different noise coefficients. Consider, for instance, the following  SDEs on $\R^d$ for $i=1,2$:
$$\d X_t^i= b_i(t,X_t^i)\d t+ \ss{a_i(t)}\d B_t,\ \ X_0^i=x\in\R^d, t\ge 0,  $$
where   $a_i(t)$  is positive definite,  and for some constant $K>1$,
$$|b_i(t,x)-b_i(t,y)|\le K|x-y|,\ \ \ K^{-1} I_d\le a_i(t)\le K I_d,\ \ x,y\in\R^d,t\ge 0.$$
Then \cite[Theorem 1.1]{BRS} gives the entropy estimate
\beg{align*}&\Ent(\L_{X_t^2}|\L_{X_t^1})\le \ff 1 2 \int_0^t \d s\int_{\R^d}  \big|a_1(s)^{-\ff 1 2 } \Phi(s,y)\big|^2\rr_2(s,y)  \d y,\\
&\Phi(s,y):= (a_1(s)-a_2(s)) \nn\log \rr_2(s,y) +b_2(s,y)-b_1(s,y),\ \ s> 0, y\in\R^d,\end{align*}
where $\rr_2(s,y):=\ff{ \L_{X_s^2}(\d y)}{\d y}$ is the distribution density function of $X_s^2$.   Since for elliptic diffusion processes
$$\int_{\R^d}  \big|\nn\log \rr_2(s,y) \big|^2\rr_2(s,y) \d y $$ behaves like   $\ff c s$ for  some constant $c>0$ and  small $s>0$, to derive finite  entropy upper bound from this estimate one may assume
\beq\label{TY} \int_0^1 \ff{\|a_1(s)-a_2(s)\|^2}s\d s <\infty,\end{equation} where $\|\cdot\|$ is the operator norm of matrices.
To  bound $\Ent(P_t^*\nu|P_t^*\mu)$ for  \eqref{EMV},  we take
$$a_1(s):= (\si_s\si_s^*)(P_s^*\mu),\ \ a_2(s):= (\si_s\si_s^*)(P_s^*\nu).$$
 But   \eqref{TY} fails when   $\|(\si_s\si_s^*)(P_s^*\mu)-(\si_s\si_s^*)(P_s^*\nu)\|$ is uniformly positive     for small  $s.$

\

 The remainder of the paper is organized as follows.   In Section 2 and Section 3, we establish the log-Harnack inequality  and Bismut formula   for the non-degenerate case and degenerate cases respectively.  In Section 4 we apply the log-Harnack inequality    to study  the exponential ergodicity  in entropy.

\section{ Non-degenerate case}

In this part, we establish the log-Harnack inequality and Bismut formula for $P_tf$ defined in \eqref{PT}, where $P_t^*\mu:=\L_{X_t^\mu}$ for $X_t^\mu$ solving   \eqref{E'} with initial distribution $\mu$.

\subsection{Log-Harnack inequality}

To establish the  log-Harnack inequality, we make   the following assumption.
\begin{enumerate} \item[{\bf (A)}] $\ll >0$ is a constant,  and there exists   $0\le K\in  L^1([0,T])$  such that
 \begin{align*}
&|b_t(x,\mu)-b_t(y,\nu)|^2+ \|\tt\si_t (\mu)-\tt\si_t (\nu)\|^2 \le  K_t (|x-y|^2+\W_2(\mu, \nu)^2),\\
&|b_t(0,\delta_0)|+\|\tt\sigma_t(\delta_0)\|^2\leq K_t,\ \  \ t\in [0,T],\ x,y\in\R^d,\ \mu,\nu\in \scr P_2(\R^d).
\end{align*}
\end{enumerate}

By \cite[Theorems 2.1 and  3.3]{HRW} or \cite[Theorem 2.1]{FYW1}, assumption {\bf (A)} implies that  the  SDE \eqref{E'} is   well-posed  for distributions in $\scr P_2(\R^d),$ and there exists a constant $c>0$ such that
\begin{align}\label{WDS}\W_2(P_t^\ast \nu,P_t^\ast \mu)\leq c\W_2(\nu, \mu),\ \ \mu,\nu\in\scr P_2(\R^d), t\in[0,T].
\end{align}

\begin{thm}\label{Loh} Assume {\bf (A)} and let $P_t$ be defined in $\eqref{PT}$ for the SDE  $\eqref{E'}$.
Then there exists a constant $c>0$ such that
\begin{equation*}
P_t\log f(\nu) \le \log P_t f(\mu) +\ff c t \W_2(\mu,\nu)^2,\ \ 0<f\in \B_b(\R^d), \mu,\nu \in \scr P_2(\R^d), t\in (0,T]. \end{equation*}
Equivalently,
$$\Ent(P_t^*\nu|P_t^*\mu) \leq \frac{c}{t}\W_2(\mu,\nu)^2,\ \ \mu,\nu \in \scr P_2(\R^d), t\in (0,T].$$
 \end{thm}
\begin{proof} As explained in Introduction, we will use coupling by change of measure under the conditional expectation given $\tt W$, which will be enough for the proof of
the log-Harnack inequality. But for the study  of Bismut formula later on, we will use the conditional probability and the conditional expectation given both $\tt W$ and $\F_0$:
$$\P^{\tt W,0}:= \P(\ \cdot\ |\tt W,\F_0),\ \ \E^{\tt W,0}:= \E(\ \cdot\ | \tt W,\F_0).$$

(a)  For any $t\in [0,T], \mu\in \scr P_2(\R^d)$ and $ f\in\B_b(\R^d),$  let
$$ P_t^{\tt W,0} f(X_0^\mu):= \E^{\tt W,0} [f(X_t^\mu)]= \E\big[f(X_t^\mu)\big|\tt W,\F_0\big],$$
where $X_t^\mu$ solves \eqref{E'} with $\L_{X_0^\mu}=\mu$.
 By \eqref{PT},
   \begin{equation}\label{PTW} P_tf(\mu)= \E \big[P_t^{\tt W,0} f(X_0^\mu)\big], \ \ \ t\in [0,T], \mu\in \scr P_2(\R^d), f\in \B_b(\R^d).\end{equation}
 Next, let
   \begin{equation}\label{XM} \xi_t^\mu:= \int_0^t \tt \si_s(P_s^*\mu)\d \tt W_s,\ \ t\in [0,T], \mu\in \scr P_2(\R^d).\end{equation}
By  {\bf (A)}, BDG's inequality and \eqref{WDS},  we find   constants $C_1,C_2>0$ such that
\beq\label{-2} \E\Big[\sup_{t\in [0,T]} |\xi_{t}^\mu-\xi_{t}^\nu|^2\Big]\le C_1  \W_2(\mu,\nu)^2\int_0^TK_s\d s\le C_2 \W_2(\mu,\nu)^2,\ \ \mu,\nu\in \scr P_2(\R^d). \end{equation}

(b) For fixed $\mu,\nu\in \scr P_2(\R^d)$, we take $\F_0$-measurable $X_0^\mu$ and $X_0^\nu$ such that
\beq\label{0} \L_{X_0^\mu}=\mu,\ \ \L_{X_0^\nu}=\nu,\ \ \E[|X_0^\mu-X_0^\nu|^2] =\W_2(\mu,\nu)^2.\end{equation}
Since $X_t^\mu$ solves \eqref{E'} with $\L_{X_0^\mu}=\mu$, we have  $\L_{X_t^\mu}=P_t^*\mu$ and the SDE becomes
\beq\label{1} \d X_t^\mu = b_t(X_t^\mu, P_t^*\mu)\d t + \ll \d W_t +\tt\si_t(P_t^*\mu)\d\tt W_t,\ \ t\in [0,T].\end{equation}
For fixed $t_0\in (0,T]$, consider the following SDE:
\beq\label{2} \beg{split}&\d Y_t = \Big\{b_t(X_t^\mu, P_t^*\mu)+\ff 1 {t_0}\big[\xi_{t_0}^\mu-\xi_{t_0}^\nu+X_0^\mu-X_0^\nu\big]\Big\} \d t + \ll \d W_t +\tt\si_t(P_t^*\nu)\d\tt W_t,\\
&\quad \ t\in [0,t_0],\ Y_0=X_0^\nu.\end{split} \end{equation}
By    \eqref{XM}, \eqref{1} and \eqref{2}, we obtain
\beq\label{3} Y_t-X_t^\mu = \ff{t_0-t}{t_0} (X_0^\nu-X_0^\mu) +\ff t{t_0} \big(\xi_{t_0}^\mu-\xi_{t_0}^\nu\big) +\xi_t^\nu-\xi_t^\mu,\ \ t\in [0,t_0].\end{equation}
 To formulate $P_{t_0} f(\nu)$ using $Y_{t_0}$, we make Girsanov's transform as follows. Let
\beq\label{ETA} \eta_t:=b_t(Y_t,P_t^*\nu)-b_t(X_t^\mu,P_t^*\mu) + \ff 1 {t_0} \big[\xi_{t_0}^\nu-\xi_{t_0}^\mu+X_0^\nu-X_0^\mu\big],\ \ t\in [0,t_0].\end{equation}  By
 {\bf (A)} and \eqref{WDS}, we find a constant $c_1>0$ such that
 \begin{equation*} \beg{split} |\eta_t|^2\le &\,c_1K_t\big( \W_2(\mu,\nu)^2 + |\xi_t^\nu-\xi_t^\mu|^2\big)\\
 &  + c_1\Big(\ff {t^2K_t+1}{t_0^2} |\xi_{t_0}^\mu-\xi_{t_0}^\nu|^2 +\ff 1 {t_0^2} |X_0^\mu-X_0^\nu|^2 \Big),\ \ t\in [0,t_0].\end{split}\end{equation*}
Since $\int_0^TK_t\d t<\infty$,  we find a constant $c_2>0$ uniform in $t_0\in (0,T]$, such that
\beq\label{*N} \ff 1 {2\ll^2} \int_0^{t_0} |\eta_t|^2\d t \le   c_2\W_2(\mu,\nu)^2+\ff{c_2}{t_0}\Big(|X_0^\mu-X_0^\nu|^2+ \sup_{t\in [0,t_0]} |\xi_t^\mu-\xi_t^\nu|^2\Big).\end{equation}
Let $\d\Q^{\tt W,0}:= R^{\tt W,0}\,\d\P^{\tt W,0}$, where
\beq\label{RW} R^{\tt W,0}:= \e^{\int_0^{t_0}\<\ff 1 \ll \eta_s, \d W _s\> -\ff 1 2 \int_0^{t_0} |\ff 1 \ll \eta_s|^2\d s}.\end{equation}
 By Girsanov's theorem,  under the weighted conditional probability $\Q^{\tt W,0},$
 $$ \hat{W}_t :=   W_t-\int_0^t \ff 1 \ll \eta_s\d s,\ \ t\in [0,t_0]$$
 is a $d$-dimensional Brownian motion. By  \eqref{2}, $\hat Y_t:= Y_t-\xi_t^\nu$ solves the SDE
 $$\d \hat Y_t = b_t(\hat Y_t+\xi_t^\nu, P_t^*\nu)\d t + \ll \d \hat{W}_t,\ \ t\in [0,t_0], \hat Y_0=X_0^\nu.$$
 On the other hand, let $X_t^\nu$ solve \eqref{E'} with initial value $X_0^\nu$. Then
 $$\hat X_t^\nu:= X_t^\nu-\xi_t^\nu,\ \ t\in [0,t_0]$$ solves the same SDE as $\hat Y_t$  for $W$ replacing $\hat W$. Then  the weak uniqueness of this equation ensured by {\bf (A)}
 implies
  $$\L_{\hat Y_{t_0} |\Q^{\tt W,0}} =\L_{\hat X_{t_0}^\nu|\P^{\tt W,0}},$$
 where  $\L_{\hat Y_{t_0} |\Q^{\tt W,0}} $ is the law of $\hat Y_{t_0}$ under $\Q^{\tt W,0}$, while    $\L_{\hat X_{t_0}^\nu|\P^{\tt W,0}}$ is the law of $\hat X_{t_0}^\nu$  under $\P^{\tt W,0}$. Since
 $\xi_{t_0}^\nu$ is deterministic given $\tt W$, it follows that
 $$\L_{Y_{t_0}|\Q^{\tt W,0}}= \L_{\hat Y_{t_0} +\xi_{t_0}^\nu|\Q^{\tt W,0}} = \L_{\hat X_{t_0}^\nu +\xi_{t_0}^\nu|\P^{\tt W,0}} = \L_{X_{t_0}^\nu |\P^{\tt W,0}}.$$
 Combining this with   $X_{t_0}^\mu=Y_{t_0}$ due to \eqref{3}, we obtain
  \beq\label{NB} P_{t_0}^{\tt W,0}   f(X_0^\nu):= \E^{\tt W,0}[   f(X_{t_0}^\nu)] = \E_{\Q^{\tt W,0}} [  f(Y_{t_0}) ]
 = \E^{\tt W,0}[ R^{\tt W,0}  f(X_{t_0}^\mu) ],\ \ f\in\B_b(\R^d).\end{equation}
 By    Young's  inequality \cite[Lemma 2.4]{ATW},  we derive
\beg{align*} &P_{t_0}^{\tt W,0}\log  f(X_0^\nu):= \E^{\tt W,0}[ \log f(X_{t_0}^\nu)] = \E_{\Q^{\tt W,0}} [\log f(Y_{t_0}) ]\\
 &= \E^{\tt W,0}[ R^{\tt W,0}  \log f(X_{t_0}^\mu) ]  \le \log \E^{\tt W,0} [ f(X_{t_0}^\mu)] + \E^{\tt W,0} [R^{\tt W,0}\log R^{\tt W,0}] \\
& = \log P_{t_0}^{\tt W,0}  f(X_{0}^\mu) + \ff 1 2 \int_0^{t_0}\ff 1 {\ll^2}\E_{\Q^{\tt W,0}} [|  \eta_t|^2]\d t,\ \ 0< f\in \B_b(\R^d).\end{align*}
This together with \eqref{*N} gives
\beq\label{*W0} P_{t_0}^{\tt W,0}\log  f(X_0^\nu)\le \log P_{t_0}^{\tt W,0}  f(X_{0}^\mu) + c_2\W_2(\mu,\nu)^2+\ff{c_2}{t_0}\Big(|X_0^\mu-X_0^\nu|^2+ \sup_{t\in [0,t_0]} |\xi_t^\mu-\xi_t^\nu|^2\Big).
\end{equation}
Taking expectation for both sides, by \eqref{PTW}, \eqref{-2}, \eqref{0} and Jensen's inequality,    we find a constant $c>0$ such that
 \beg{align*} P_{t_0}\log f(\nu)&= \E\big[P_{t_0}^{\tt W,0}\log  f(X_{0}^\nu)\big] \le \E\big[ \log P_{t_0}^{\tt W,0}  f(X_{0}^\mu)\big]+ \ff{c}{t_0}   \W_2(\mu,\nu)^2\\
 &\le \log  P_{t_0}    f(\mu) + \ff{c}{t_0}   \W_2(\mu,\nu)^2,\ \ t_0\in (0,T], \mu,\nu\in \scr P_2(\R^d).\end{align*}
 \end{proof}

\subsection{Bismut formula}

 We aim to establish
the Bismut type formula \eqref{BS} for the intrinsic derivative of $P_tf$. To this end, we first recall the definition of intrinsic derivative, see \cite{RW0} for historical remarks  on this derivative and    links to other derivatives for functions of   measures.

 \begin{defn}   Let $k\in (1,\infty)$. \beg{enumerate} \item[$(1)$] A continuous function
$f$ on $\scr P_k(\R^d)$ is called intrinsically differentiable, if for any $\mu\in \scr P_k(\R^d)$,
 $$T_{\mu,k}(\R^d):=L^k(\R^d\to\R^d;\mu)\ni\phi\mapsto D_\phi^If(\mu):= \lim_{\vv\downarrow 0} \ff{f(\mu\circ(id+\varepsilon\phi)^{-1})-f(\mu)}{\vv}\in\mathbb{R}
$$ is a well  defined  bounded linear operator. In this case,   the norm of    the intrinsic derivative $D^I f(\mu)$  is given by
$$\|D^If(\mu)\|_{L^{k^*}(\mu)} :=\sup_{\|\phi\|_{L^k(\mu)}\le 1} |D^I_\phi f(\mu)|.$$
\item [$(2)$] $f$ is called $L$-differentiable on $\scr P_k(\R^d)$, if it is intrinsically differentiable and
 \begin{equation*}
 \lim_{\|\phi\|_{T_{\mu,k}(\R^d)}\downarrow0}\ff{|f(\mu\circ(id+\phi)^{-1})-f(\mu)-D_\phi^If(\mu)| }{\|\phi\|_{T_{\mu,k}(\R^d)}}=0,\ \ \mu\in \scr P_k(\R^d).
 \end{equation*}
We denote $f\in C^1(\scr P_k(\R^d))$, if it is $L$-differentiable such that $D^If(\mu)(x)$ has a jointly continuous version in $(x,\mu)\in \R^d\times \scr P_k(\R^d)$.
 \item[$(3)$] We denote  $g\in C^{1,1}(\R^d\times \scr P_k(\R^d))$, if $g: \R^d\times\scr P_k(\R^d)\to\R$  is $C^1$ in $x\in \R^d$ and $\mu\in \scr P_k(\R^d)$ respectively,  such that
 $$\nn g(x,\mu):= \nn \{g(\cdot,\mu)\}(x),\ \ \ D^Ig(x,\mu)(y):= D^I\{g(x,\cdot)\}(\mu)(y)$$ are jointly continuous in $(x,y,\mu)\in \R^d\times\R^d\times\scr P_k(\R^d)$.
 \end{enumerate}
  \end{defn}

 In this part, we consider \eqref{E'} with coefficients
 $$\tilde{\si}: [0,T]\times  \scr P_k(\R^d)\to \R^d\otimes\R^d,\ \ b: [0,T]\times \R^d\times \scr P_k(\R^d)\to\R^d$$
 satisfying the following assumption.

\beg{enumerate} \item[{\bf (B)}]     $\ll>0$ and $k\in (1,\infty)$ are constants,       denote $k^*:=\ff{k}{k-1}$.
 For any $t\in [0,T]$,   $b_t\in C^{1,1}(\R^d\times\scr P_k(\R^d)), $   $\tt\sigma_t\in C^1(\scr P_k(\R^d))$, and there exists    $0\le K\in L^1([0,T])$ such that
 \beg{align*}
 &  |D^I b_t(x,\cdot)(\mu)(y)|+\|D^I  \tt\sigma_t  (\mu)(y) \| \le \ss{K_t} (1+|y|^{k-1}) ,\\
 &    |b_t(0,\delta_0)|+   |\nn b_t(\cdot,\mu)(x)| \le \ss {K_t},\  \ (t,x,\mu)\in [0,T]\times  \R^d\times\scr P_k(\R^d),\  y\in\R^d.\end{align*}
   \end{enumerate}

By \cite[Lemma 3.1]{FYW3}, {\bf (B)} implies     {\bf (A)}    for $(\scr P_k(\R^d),\W_k)$ replacing  $(\scr P_2(\R^d),\W_2).$
 So,  according to \cite[Theorem 3.3]{HRW},   the SDE \eqref{E'} is well-posed for distributions in $\scr P_k(\R^d)$, and there exists a constant $c>0$ such that
\begin{align}\label{WBS}\W_k(P_t^\ast \mu,P_t^\ast \nu)\leq c  \W_k(\mu,\nu),\ \ \mu,\nu\in\scr P_k(\R^d), t\in[0,T].
\end{align} By this estimate and   {\bf (A)}    for $(\scr P_k(\R^d),\W_k)$ replacing  $(\scr P_2(\R^d),\W_2),$  the argument leading to \eqref{*W0} yields that there exists a constant $c>0$ such that for any $t\in(0,T], 0<f\in\scr B_b(\R^d)$,
\beq\label{*W}P_{t}^{\tt W,0}\log  f(X_0^\nu)\le \log P_{t}^{\tt W,0}  f(X_{0}^\mu)+c \W_k(\mu,\nu)^2+\ff{c }{t}\Big(|X_0^\mu-X_0^\nu|^2+ \sup_{s\in [0,t]} |\xi_s^\mu-\xi_s^\nu|^2\Big).\end{equation}

To calculate $D_\phi^IP_tf(\mu)$ for $\mu\in \scr P_k(\R^d)$ and $\phi\in T_{\mu,k}(\R^d)$,     let  $X_0^\mu$ be $\F_0$-measurable such that $\L_{X_0^\mu}=\mu.$ Then
  $$\L_{X_0^{\mu}+\vv\phi(X_0^\mu)}=\mu^\vv:= \mu\circ (id+\vv\phi)^{-1},\ \ \ \vv\in [0,1].$$
  For any $\vv\in [0,1],$ let
  $X_t^{\mu^\vv}$ solve \eqref{E'} with $ X_0^{\mu^\vv}=X_0^{\mu}+\vv\phi(X_0^\mu),$  i.e.
   \begin{equation*}\beg{split}& \d X_t^{\mu^\vv} = b_t(X_t^{\mu^\vv}, P_t^*\mu^\vv) \d t +\ll \d W_t+ \tt\si_t(P_t^*\mu^\vv) \d \tt W_t,\\
 &\ X_0^{\mu^\vv}=X_0^\mu+\vv\phi(X_0^\mu), t\in [0,T],\vv\in [0,1].\end{split} \end{equation*}
  Consider the spatial derivative of $X_t^{\mu}$ along $\phi$:
 \begin{equation*}  \nn_\phi X_t^{\mu}:=\lim_{\vv\downarrow 0} \ff{X_t^{\mu^\vv}-X_t^{\mu}}\vv,\ \
 \  \ t\in [0,T], \phi\in T_{\mu,k}(\R^d).\end{equation*}
  For any $0\le s<t\le T,$ define
\beg{align*} &N^{\mu,\phi}_{s,t}:=\ff{t-s} t\phi(X_0^\mu)+\int_0^s\Big\<\E\big[\<D^I\tilde{\si}_r(P_r^*\mu)(X_r^\mu), \nn_\phi X_r^\mu\>\big],\ \d\tt W_r\Big\>\\
&\qquad \qquad -\ff s t \int_0^t \Big\<\E\big[\<D^I\tilde{\si}_r(P_r^*\mu)(X_r^\mu), \nn_\phi X_r^\mu\>\big],\ \d\tt W_r\Big\>,\\
&M^{\mu,\phi}_{s,t}:=\E\big[\big\< \{D^Ib_s(y, \cdot)\}(P_s^\ast \mu)(X_s^\mu),\nabla_{\phi}X_s^\mu\big\>\big]_{y=X_s^{\mu}}+ \ff 1 {t}\phi(X_0^\mu)\\
&\qquad\qquad +\ff 1 {t} \int_0^t \Big\<\E\big[\<D^I\tilde{\si}_r(P_r^*\mu)(X_r^\mu), \nn_\phi X_r^\mu\>\big],\ \d\tt W_r\Big\>.\end{align*}
 The main result in this part is the following.

\beg{thm}\label{TA2'} Assume   {\bf (B)}.
\beg{enumerate}
\item[$(1)$] For any $\mu\in \scr P_k(\R^d)$ and $\phi\in T_{\mu,k}(\R^d)$, $(\nn_\phi X_\cdot^{\mu})$ exists in $L^k(\Omega\to C([0,T],\R^d),\P)$ such that for some    constant $c>0$,
$$\E\Big[\sup_{t\in[0,T]}|\nabla_{\phi}X_t^\mu|^k\Big]\leq c\|\phi\|_{L^k(\mu)}^k,\ \ \mu\in \scr P_k(\R^d), \phi\in T_{\mu,k}(\R^d).$$
\item[$(2)$]
 For any $f\in \B_b(\R^d), t\in (0,T]$,   $\mu\in \scr P_k(\R^d)$ and $\phi\in T_{\mu,k}(\R^d)$, $D^I_\phi P_tf(\mu)$ exists and satisfies
 \beq\label{BSMI}     D_\phi^I  P_tf(\mu)=\ff 1 \ll    \E\bigg[f(X_t^{\mu})\int_0^t \Big\< \nn_{N_{s,t}^{\mu,\phi}}b_s(\cdot, P_s^*\mu)(X_s^\mu)  +M_{s,t}^{\mu,\phi},\ \d W_s\Big\>\bigg].
    \end{equation}
Consequently, $P_tf$ is intrinsically differentiable and for some
 constant $c>0$,
\beq\label{EST}\beg{split}& \|D^I P_t f(\mu)\|_{L^{k^*}(\mu)}
 \le \ff c {\ss t}\big( P_t |f|^{k^*}(\mu)\big)^{\ff 1 {k^*}},\\
 &\qquad\quad    f\in \B_b(\R^d),  \mu\in \scr P_k(\R^d), t\in (0,T].\end{split}\end{equation}\end{enumerate}
\end{thm}

\beg{proof}  The first assertion follows from \cite[Lemma 5.2]{BRW}.  By the first assertion, {\bf (B)} and the definition  of $(N_{s,t}^{\mu,\phi}, M_{s,t}^{\mu,\phi})$,  we deduce
\eqref{EST} from \eqref{BSMI}. So, it remains to prove \eqref{BSMI}.

(a) Since {\bf (B)} implies {\bf (A)} for $\scr P_k(\R^d)$ replacing $\scr P_2(\R^d)$, the argument in the proof of Theorem \ref{Loh} up to \eqref{NB} still applies.
For fixed $t_0\in (0,T], \mu\in \scr P_k(\R^d)$ and $\phi\in T_{\mu,k}(\R^d)$, let $X_t^\mu$ solve \eqref{1}. Next, for any $\vv\in (0,1]$, let
$ Y_t^\vv$ solve \eqref{2} for
$$\nu=\mu^\vv,\ \ Y_0=Y_0^\vv:= X_0^\mu+\vv\phi(X_0^\mu).$$
Then \eqref{3} with $(Y_t,\nu)=(Y_t^\vv,\mu^\vv)$ becomes
\beq\label{YX} Y_t^\vv-X_t^\mu= \ff{t_0-t}{t_0}\vv\phi(X_0^\mu) +\ff t{t_0} (\xi_{t_0}^\mu-\xi_{t_0}^{\mu^\vv}) +\xi_t^{\mu^\vv}-\xi_t^\mu,\ \ t\in [0,t_0].\end{equation}
Let
$$H_t:= \int_0^t \Big\<\E\big[\big\<D^I\tilde{\si}_s(P_s^*\mu)(X_s^\mu), \nn_\phi X_s^\mu\big\>\big],\ \d\tt W_s\big\>,\ \ t\in [0,T].$$
By {\bf (B)} and \eqref{WBS}, we obtain
\beq\label{a1} \big\|\tt\si_s(P_s^*\mu^\vv)-\tt\si_s(P_s^*\mu)\big\|^2\le \vv^2 c^2 K_s \|\phi\|_{L^k(\mu)}^2,\ \ \vv\in [0,1], s\in [0,T].\end{equation}
So, by {\bf (B)},   the chain rule in  \cite[Theorem 2.1(1)]{BRW}, \eqref{XM}, BDG's inequality and the dominated convergence theorem, we obtain
\beq\label{a2} \lim_{\vv\downarrow 0} \E\bigg[\sup_{t\in [0,T]} \Big|\ff{\xi_t^{\mu^\vv}-\xi_t^\mu}\vv -H_t\Big|^2\bigg]=0.\end{equation}
Let $(\eta_t^\vv, R^\vv)=(\eta_t, R^{\tt W,0})$ be defined in \eqref{ETA} and \eqref{RW} for $(Y_t,\nu)=(Y_t^\vv,\mu^\vv)$. By
 {\bf (B)} and \eqref{YX}, we find a constant $\kk>0$ such that
 \beg{align*}& \ff{|\eta_s^\vv|^2}{\vv^2} \le\kk K_s \bigg( \|\phi\|_{L^k(\mu)}^2 +|\phi(X_0^\mu)|^2+\sup_{t\in [0,t_0]} \ff{|\xi_t^{\mu^\vv}-\xi_t^\mu|^2}{\vv^2}\bigg)=:\LL_s,\\
 &\lim_{\vv\downarrow 0} \ff{\eta_s^\vv}\vv =  \nn_{N_{s,t_0}^{\mu,\phi}}b_s(\cdot, P_s^*\mu)(X_s^\mu)+M_{s,t_0}^{\mu,\phi},\ \ s\in [0,t_0].\end{align*}
 Since $\LL_s$ is deterministic given $\tt W$ and $\F_0$, this together with \eqref{NB} and the dominated convergence theorem yields
 \beq\label{*W2}\beg{split}&  \lim_{\vv\downarrow 0} \ff{P_{t_0}^{\tt W,0} f(X_0^{\mu^\vv})-P_{t_0}^{\tt W,0} f(X_0^\mu)}\vv =\lim_{\vv\downarrow 0} \E^{\tt W,0}\Big[f(X_{t_0}^\mu)\ff{R^\vv-1}\vv\Big]\\
 &=   \frac{1}{\lambda}\E^{\tt W,0} \bigg[f(X_{t_0}^\mu)\int_0^{t_0} \Big\< \nn_{N_{s,t_0}^{\mu,\phi}}b_s(\cdot, P_s^*\mu)(X_s^\mu) + M_{s,t_0}^{\mu,\phi},\ \d W_s\Big\>\bigg].\end{split}\end{equation}

(b) Let $\L_{\xi|\P^{\tt W,0}}$ be the conditional distribution of a random variable $\xi$ under $\P^{\tt W,0}$.
By Pinsker's inequality and \eqref{*W}, we have
\beg{align*}& \sup_{\|f\|_\infty\le 1} \big|P_{t_0}^{\tt W,0} f(X_0^{\mu^\vv})-P_{t_0}^{\tt W,0} f(X_0^\mu)\big|^2 \le 2\,   \Ent\big(\L_{X_{t_0}^{\mu^\vv}|\P^{\tt W,0}}\big|\L_{X_{t_0}^\mu|\P^{\tt W,0}}\big)\\
&\le  c \W_k(\mu^\vv,\mu)^2+ \ff{c}{t_0}\Big(\vv^2|\phi(X_0^\mu)|^2 + \sup_{t\in [0,t_0]} |\xi_t^\mu-\xi_t^{\mu^\vv}|^2\Big).\end{align*}
This together with $\W_k(\mu^\vv,\mu)\le \vv \|\phi\|_{L^k(\mu)}$ implies that for some constant $c(t_0)>0$,
\beg{align*} &\ff{|P_{t_0}^{\tt W,0} f(X_0^{\mu^\vv})-P_{t_0}^{\tt W,0} f(X_0^\mu)|}{\vv} \\
&\le  \|f\|_\infty  c(t_0)\Big(\|\phi\|_{L^k(\mu)}+ |\phi(X_0^\mu)|  + \sup_{t\in [0,t_0]}\ff{ |\xi_t^{\mu^\vv}-\xi_t^\mu|}\vv\Big), \ \  \vv\in (0,1].\end{align*}
Combining this with \eqref{-2} and \eqref{a1}, we may apply the dominated convergence theorem to \eqref{*W2} to derive
\beg{align*} &D^I_\phi P_{t_0} f(\mu):=\lim_{\vv\downarrow 0} \E\bigg[\ff{P_{t_0}^{\tt W,0} f(X_0^{\mu^\vv})-P_{t_0}^{\tt W,0} f(X_0^\mu)}\vv\bigg]
 =  \E\bigg[\lim_{\vv\downarrow 0} \ff{P_{t_0}^{\tt W,0} f(X_0^{\mu^\vv})-P_{t_0}^{\tt W,0} f(X_0^\mu)}\vv\bigg]\\
&=  \frac{1}{\lambda}\E \bigg[f(X_{t_0}^\mu)\int_0^{t_0} \Big\< \nn_{N_{s,t_0}^{\mu,\phi}}b_s(\cdot, P_s^*\mu)(X_s^\mu) + M_{s,t_0}^{\mu,\phi},\ \d W_s\Big\>\bigg].\end{align*}
\end{proof}

\section{Degenerate case }

 Consider the following distribution dependent stochastic Hamiltonian system for $X_t=(X_t^{(1)}, X_t^{(2)})\in \mathbb{R}^{m+d}$:
\beq\label{E0}
\begin{cases}
\d X_t^{(1)}=\big\{AX^{(1)}_t+MX_t^{(2)}\big\}\d t, \\
\d X_t^{(2)}=b_t(X_t,\L_{X_t})\d t+\sigma_t(\L_{X_t})\d B_t,\ \ t\in [0,T],
\end{cases}
\end{equation}
where $B=(B_t)_{t\in [0,T]}$ is a $d$-dimensional standard Brownian motion, $A$ is an $m\times m$ and $M$ is an $m\times d$ matrix, and
$$\sigma:[0,T]\times \scr P (\R^{m+d})\to \mathbb{R}^{d}\otimes\R^d,\ \ b:[0,T]\times\R^{m+d}\times \scr P(\R^{m+d})\to\mathbb{R}^d$$ are measurable, where
$\scr P(\R^{m+d})$ is the space of probability measures on $\R^{m+d}$ equipped with the weak topology.
 In \cite{GW,WZ1}, where the coefficients are distribution independent, the Bismut formula is derived for stochastic Hamiltonian system.

For any $k\ge 1$, let
$$\scr P_k(\R^{m+d}):=\big\{\mu\in \scr P(\R^{m+d}):\ \|\mu\|_k:=\mu(|\cdot|^k)^{\ff 1 k}<\infty\big\},$$
which is a Polish space under the $L^k$-Wasserstein distance $\W_k$. When \eqref{E0} is well-posed for distributions in $\scr P_k(\R^{m+d})$, let
$P_t^*\mu=\L_{X_t}$ for the solution with initial distribution $\mu\in \scr P_k(\R^{m+d})$. We
  aim to establish   the log-Harnack inequality and Bismut formula for
  $$P_tf(\mu):= \int_{\R^{m+d}}f\d (P_t^*\mu),\ \ f\in \B_b(\R^{m+d}).$$

   By the same reason reformulating \eqref{EMV} as \eqref{E'},  instead of $\eqref{E0}$ we consider
 \beq\label{E00}
\begin{cases}
\d X_t^{(1)}=\big\{AX^{(1)}_t+MX_t^{(2)}\big\}\d t, \\
\d X_t^{(2)}=b_t(X_t,\L_{X_t})\d t+\ll \d W_t+ \tt\sigma_t(\L_{X_t})\d \tt W_t,\ \ t\in [0,T],
\end{cases}
\end{equation} where   $W_t,\tt W_t$ are two independent $d$-dimensional Brownian  motions, and
$$\tt \si: [0,T]\times \scr P(\R^{m+d})\to \R^d\otimes \R^d$$ are measurable.

\subsection{Log-Harnack inequality}

To establish the log-Harnack inequality, we make the following assumption.

 \begin{enumerate}
\item[\bf{(C)}] $\ll>0$ is a constant,  $(\tt \si,b)$ satisfies conditions in {\bf (A)} for $(x,\mu)\in \R^{m+d}\times \scr P_2(\R^{m+d})$, and    the following  Kalman's rank condition holds for some integer $1\leq l\leq m$:
\begin{align}\label{RRS}\mathrm{Rank}[A^iM, 0\le i\le l-1]=m,\end{align} where $A^0:=I_m$ is the $m\times m$-identity matrix. \end{enumerate}
 By \cite[Theorem 2.1]{FYW1},   {\bf (C)} implies that   \eqref{E00} is   well-posed  for distributions in $\scr P_2(\R^{m+d}),$ and there exists a constant $c>0$ such that
$$\W_2(P_t^*\mu,P_t^*\nu)\le c\W_2(\mu,\nu), \ \ \ \mu,\nu\in \scr P_2(\R^{m+d}), t\in [0,T].$$ So, as in \eqref{-2}, we find a constant $C>0$ such that
\beq\label{*N1} \E\bigg[\sup_{t\in [0,T]} |\xi_t^\mu-\xi_t^\nu|^2\bigg] \le C\W_2(\mu,\nu)^2,\ \ \mu,\nu\in \scr P_2(\R^{m+d}).\end{equation}

To distinguish the singularity of $P_t$ in the degenerate component $x^{(1)}$ and the non-degenerate one $x^{(2)}$,
for any $t>0$ we consider the modified distance
$$\rr_t(x,y):= \ss{t^{-2}|x^{(1)}-y^{(1)}|^2+|x^{(2)}-y^{(2)}|^2},\ \ \ x,y\in\R^{m+d},$$
and define the associated $L^2$-Wasserstein distance
$$\W_{2,t}(\mu,\nu):=\inf_{\pi\in \C(\mu,\nu)}\bigg( \int_{\R^{m+d}\times \R^{m+d}} \rr_t(x,y)^2 \pi(\d x,\d y)\bigg)^{\ff 1 2}. $$
It is clear that
\beq\label{W0} \ff 1 {T^2\lor 1} \W_2^2\le \W_{2,t}^2 \le \ff{1\lor T^2}{t^2} \W_2^2,\ \ \ \ t\in (0,T].\end{equation}
For    $t\in (0,T]$, let
$$Q_{t} :=\int_0^{t} \ff{s(t-s)}{t^2}\e^{-sA}MM^*\e^{-sA^*}\d s.$$
According to \cite{S}, see also \cite[Proof of Theorem 4.2(1)]{WZ1},  the rank condition  \eqref{RRS} implies
\beq\label{QT} \|Q_t^{-1}\|\le c_0 t^{1-2l},\ \ \ t\in (0,T]\end{equation}
for some constant $c_0>0$.

\begin{thm}\label{LHI} Assume {\bf (C)} and let    $P_t^*$ be associated with the degenerate  SDE  $\eqref{E00}.$
Then there exists a constant $ c>0$ such that
\beq\label{EC2}  \beg{split}&P_t\log f(\nu) -   \log P_tf(\mu)
 \le \ff{c}{t^{4l-3}}\W_{2,t}(\mu,\nu)^2  \le \ff{c(1\lor T^2)}{t^{4l-1}}\W_{2}(\mu,\nu)^2,\\
 &t\in (0,T],\ \mu,\nu\in \scr P_2(\R^{m+d}),\  \ 0<f\in \B_b(\R^{m+d}).\end{split}\end{equation}
Equivalently,  for any $  t\in (0,T]$ and $  \mu,\nu\in \scr P_2(\R^{m+d}),$
$$  \Ent(P_t^*\nu|P_t^*\mu)\le   \ff{c}{t^{4l-3}}\W_{2,t}(\mu,\nu)^2  \le \ff{c(1\lor T^2)}{t^{4l-1}}\W_{2}(\mu,\nu)^2.$$
 \end{thm}

 \begin{proof}  For any $t_0\in (0,T]$ and $\mu,\nu\in \scr P_2(\R^{m+d})$, let $X_0,Y_0$ be $\F_0$-measurable such that
 \beq\label{0'} \L_{X_0}=\mu,\ \ \L_{Y_0}=\nu,\ \ \E[\rr_{t_0}(X_0,Y_0)^2]=\W_{2,t_0}(\mu,\nu)^2.\end{equation}
  Let
 $X_t $ solve \eqref{E00} with initial value $X_0$, we have $P_t^*\mu= \L_{X_t}.$ Let
 \beq\label{TTX} v= (v^{(1)}, v^{(2)}):=(Y_0^{(1)}-X_0^{(1)}, Y_0^{(2)}- X_0^{(2)})=Y_0-X_0.  \end{equation}
 For fixed $t_0\in (0,T],$ let
\beq\label{A1}\beg{split} &\aa_{t_0}(s) := \ff{s}{t_0} \big(\xi_{t_0}^\mu-\xi_{t_0}^\nu -v^{(2)}\big)
  - \ff{s(t_0-s)}{t_0^2} M^*\e^{-sA^*}Q_{t_0}^{-1}\big(v^{(1)}+ V_{t_0}^{\mu,\nu}\big),\\
  &V_{t_0}^{\mu,\nu}:= \int_0^{t_0}\e^{-rA}M\Big\{ \ff{t_0-r}{t_0} v^{(2)} +\ff r {t_0}\big(\xi_{t_0}^\mu-\xi_{t_0}^\nu\big) +\xi_r^\nu-
\xi_r^\mu\Big\}\d r.\end{split} \end{equation}
By \eqref{QT}, we find a constant $c_1>0$ independent of $t_0\in (0,T]$ such that
\beq\label{AP2} \sup_{t\in [0,t_0]} \big\{t_0|\aa_{t_0}'(t)|+|\aa_{t_0}(t)|\big\}\le \ff{c_1}{t_0^{2(l-1)}}  \Big(t_0^{-1}|v^{(1)}|+  |v^{(2)}| + \sup_{t\in [0,t_0]} |\xi_t^\mu-\xi_t^\nu|\Big).\end{equation}
Let $Y_t$ solve the SDE with initial value $Y_0$:
\beq\label{A2} \begin{cases}
\d Y_t^{(1)}=\big\{AY_t^{(1)}  +MY_t^{(2)}\big\}\d t, \\
\d Y_t^{(2)}=\big\{b_t(X_t,P_t^*\mu)+  \aa_{t_0}'(t) \big\}\d t+\ll\d W_t+\tt\si_t(P_t^*\nu)\d\tt W_t,\ \ t\in [0,t_0].
    \end{cases}\end{equation}
This and \eqref{E00} imply
\beq\label{A3}\beg{split}  &  Y_t^{(2)}-  X_t^{(2)} =   \aa_{t_0} (t)+v^{(2)} +\xi_t^\nu-\xi_t^\mu,\\
& Y_t^{(1)}- X_t^{(1)} =  \e^{tA} v^{(1)}  +   \int_0^t \e^{(t-s)A} M \big\{ \aa_{t_0}(s)+v^{(2)} +\xi_s^\nu-\xi_s^\mu\big\} \d s,\ \ \ t\in [0,t_0].\end{split}\end{equation}
Consequently,
$$Y_{t_0}^{(2)}-X_{t_0}^{(2)} = \xi_{t_0}^\mu-\xi_{t_0}^\nu-v^{(2)}+v^{(2)} +\xi_{t_0}^\nu-\xi_{t_0}^\mu=0,$$
\beg{align*} &Y_{t_0}^{(1)}-X_{t_0}^{(1)}=   \e^{t_0A} v^{(1)} +\int_0^{t_0} \e^{(t_0-s)A}M
\Big\{\ff s {t_0} \big(\xi_{t_0}^\mu-\xi_{t_0}^\nu -v^{(2)}\big)+v^{(2)}+\xi_s^\nu-\xi_s^\mu\Big\}\d s\\
&\quad - \e^{t_0A} Q_{t_0}Q_{t_0}^{-1} \bigg(v^{(1)} +\int_0^{t_0} \e^{-rA}M \Big\{\ff{t_0-r}{t_0} v^{(2)} +\ff r {t_0} \big(\xi_{t_0}^\mu-\xi_{t_0}^\nu\big) +\xi_r^\nu-\xi_r^\mu\Big\}\d r\bigg)\\
&=0, \end{align*}
so that
\beq\label{AP4} Y_{t_0}=X_{t_0}.\end{equation}
On the other hand, by \eqref{A3} and  \eqref{AP2} we find a constant $c_2>0$ uniform in $t_0\in (0,T]$ such that
\beq\label{AP3}\beg{split}&\sup_{t\in [0,t_0]}  |Y_t -X_t |^2  \le \ff{c_2}{t_0^{4(l-1)} }   \Big\{ t_0^{-2}|v^{(1)}|^2+  |v^{(2)}|^2+\sup_{t\in [0,t_0]}|\xi_t^\mu-\xi_t^\nu|^2\Big\} \\
&=\ff{c_2}{ t_0^{4(l-1)} }\Big\{\rr_{t_0}(X_0,Y_0)^2+  \sup_{t\in [0,t_0]}|\xi_t^\mu-\xi_t^\nu|^2\Big\}.  \end{split} \end{equation}
  To formulate the  equation of $Y_t$ as \eqref{E00}, let
\beq\label{ETS} \eta_s := \ff 1 \ll \big\{b_s(Y_s, P_s^*\nu)- b_s(X_s, P_s^*\mu)  -  \aa_{t_0}'(s)\big\},\ \ s\in [0,t_0].\end{equation}
By {\bf (C)}, \eqref{AP2} and \eqref{AP3}, we find a constant $c_3>0$ uniformly in $t_0\in (0,T]$  such that
\beq\label{AP5}  \beg{split}  |\eta_s |^2\le &\,c_3    K_s
\Big\{\W_2(\mu,\nu)^2  + t_0^{4(1-l)}   \rr_{t_0}(X_0,Y_0)^2   + t_0^{4(1-l)} \sup_{t\in [0,t_0]} |\xi_t^\mu-\xi_t^\nu|^2\Big\}\\
&  +  c_3 t_0^{2-4l}\Big(\rr_{t_0}(X_0,Y_0)^2+\sup_{t\in [0,t_0]} |\xi_t^\nu-\xi_t^\mu|^2\Big).\end{split}\end{equation}
By Girsanov's theorem,
$$\hat{W}_t:=W_t- \int_0^t \eta_s \d s,\ \ t\in [0,t_0]$$
is a $d$-dimensional Brownian motion under the weighted  conditional probability measure $\d \Q^{\tt W,0}:=R^{\tt W,0} \d\P^{\tt W,0}$, where
$$R^{\tt W,0}:= \e^{\int_0^{t_0} \<\eta_s,\d W_s\>-\ff 1 2 \int_0^{t_0} |\eta_s|^2\d s}.$$
 Let $\tt \xi_t^\nu=(0,\xi_t^\nu).$ By  \eqref{A2}, $\hat Y_t:= Y_t-\tt \xi_t^\nu$ solves the SDE
$$\begin{cases}
\d \hat Y_t^{(1)}=\big\{A\hat Y_t^{(1)}  +M\hat Y_t^{(2)}+M\xi_t^\nu\big\}\d t, \\
\d \hat Y_t^{(2)}= b_t(\hat Y_t+\tt \xi_t^\nu ,P_t^*\nu) \d t+\ll\d \hat{W}_t,\ \ t\in [0,t_0],\ \hat Y_0=Y_0.
\end{cases}$$
 Letting  $X_t^\nu$ solve  \eqref{E00} with $X_0^\nu=Y_0$, we see that $\hat X_t^\nu:= X_t^\nu-\tt\xi_t^\nu$  solves the same  equation  as $\hat{Y}_t$  for $W_t$ replacing $\hat{W}_t$.
By the weak uniqueness and   \eqref{AP4},  \eqref{NB} holds for $\R^{m+d}$ replacing $\R^d$, i.e.   for any $f\in \B_b(\R^{m+d})$,
\beq\label{NB'} P_{t_0}^{\tt W,0} f(X_{t_0}^\nu):= \E^{\tt W,0} [f(X_{t_0}^\nu)]= \E^{\tt W,0} [R^{\tt W,0} f(Y_{t_0})] = \E^{\tt W,0} [R^{\tt W,0} f(X_{t_0})].\end{equation}
 Combining this with Young's inequality and \eqref{AP5}, we find   constants $c_4>0$ such that
\beq\label{YPP} \beg{split} &P_{t_0}^{\tt W,0} \log f(X_{0}^\nu)- \log P_{t_0}^{\tt W,0} f(X_0^\mu) \le  \E^{\tt W,0} [R^{\tt W,0}\log R^{\tt W,0}] = \ff 1 2\E_{\Q^{\tt W,0}} \int_0^{t_0}|\eta_t|^2\d t\\
&\le c_4 \Big\{\W_2(\mu,\nu)^2  +  t_0^{3-4l}   \rr_{t_0}(X_0,Y_0)^2+ t_0^{3-4l}\sup_{t\in [0,t_0]} |\xi_t^\nu-\xi_t^\mu|^2\Big\}.  \end{split} \end{equation}
By taking expectation, using Jensen's inequality, \eqref{*N1}, \eqref{W0} and \eqref{0'},  we prove \eqref{EC2}.
  \end{proof}

\subsection{Bismut formula}
We will use Definition 2.1 for $\R^{m+d}$ replacing $\R^d$.
The following assumption is parallel to {\bf (B)} with an additional rank condition.

 \beg{enumerate} \item[{\bf (D)}]  $(\tt\si,b)$ satisfies {\bf (B)} for $\R^{m+d}$ replacing $\R^d$, and the rank condition  \eqref{RRS}   holds  for some $1\le l\le m.$      \end{enumerate}

Let   $X_0^\mu$ be $\F_0$-measurable such that $\L_{X_0^\mu}=\mu\in\scr P_{k}(\R^{m+d})$, and let $X_t^\mu$ solve \eqref{E00} with initial value $X_0^\mu$. For any $\vv\ge 0$, denote
$$\mu^\vv:= \mu\circ(id+\vv\phi)^{-1},\ \ X_0^{\mu^\vv}:= X_0^\mu+\vv\phi(X_0^\mu).$$
Let $X_t^{\mu^\vv}$ solve \eqref{E00} with initial value $X_0^{\mu^\vv}$. So,
$$X_t^{\mu}=X_t^{\mu^0},\ \ P_t^*\mu^\vv=\L_{X_t^{\mu^\vv}},\ \ t\in [0,T], \vv\ge 0.$$ By \cite[Lemma 5.2]{BRW},
  for any $\phi=(\phi^{(1)},\phi^{(2)})\in T_{\mu,k}(\R^{m+d}),$  {\bf (D)} implies that
$$\nabla_{\phi}X_\cdot^\mu:=\lim_{\vv\downarrow 0}\ff{X_\cdot^{\mu^\vv}-X_\cdot^\mu}\vv$$
exists in $L^k(\OO\to C([0,T]; \R^{m+d});\P)$, and there exists a constant $c>0$  such that
\begin{equation*}\E\bigg[\sup_{t\in [0,T]}|\nn_\phi X_t^\mu|^k\bigg]\le c \|\phi\|_{L^k(\mu)}^k,\ \ \mu\in \scr P_k(\R^{m+d}), \phi\in T_{\mu,k}(\R^{m+d}).\end{equation*}
Finally, for any $t\in (0,T]$ and $s\in [0,t]$, let
\beg{align*} & \gg_t^{\mu,\phi}:= \int_0^t \Big\<\E\big[\big\<D^I\tilde{\si}_r(P_r^*\mu)(X_r^\mu),\ \nn_\phi X_r^\mu\big\>\big],\ \d\tt W_r\Big\>\\
&V_t^{\mu,\phi}:= \int_0^t \e^{-rA}M\Big\{ \ff{t-r}t \phi^{(2)}(X_0^\mu) -\ff r t \gg_{t}^{\mu,\phi} +\gg_r^{\mu,\phi}\Big\}
\d r,\\
&\aa_t^{\mu,\phi}(s):= -\ff{s}t \big\{\phi^{(2)}(X_0^\mu) +\gg_{t}^{\mu,\phi}\big\}   -\ff{s(t-s)}{t^2} M^*\e^{-sA^*}Q_{t}^{-1} \big\{\phi^{(1)}(X_0^\mu) +V_t^{\mu,\phi}\big\},\end{align*} and define
\beg{align*}
 &N_{s,t}^{(1)}:= \e^{sA}\phi^{(1)}(X_0^\mu)+\int_0^s\e^{(s-r)A}M\big\{ \aa_t^{\mu,\phi} (r) +\phi^{(2)}(X_0^\mu) +\gg_r^{\mu,\phi}\big\}\d r\\
   &N_{s,t}^{(2)}:=  \aa_t^{\mu,\phi}(s) +\phi^{(2)}(X_0^\mu) +\gg_s^{\mu,\phi},\\
& M_{s,t}^{\mu,\phi}:= \E\big[\<D_\phi^I b_s(z,\cdot)(P_s^*\mu)(X_s^\mu),\nn_\phi X_s^\mu\>]\big]_{z=X_s^\mu}-(\aa_t^{\mu,\phi})'(s).\end{align*}
Then we have the following result.

\beg{thm}\label{TA3} Assume   {\bf(D)} and let $N_{s,t}^{\mu,\phi} :=\big(N_{s,t}^{(1)}, N_{s,t}^{(2)}\big)\in \R^{m+d}, 0\le s\le t.$  For any $t\in (0,T]$, $\mu\in \scr P_k(\R^{m+d}),$  $\phi\in T_{\mu,k}(\R^{m+d})$ and $f\in \B_b(\R^{m+d})$,
\beq\label{BS3} D_\phi^I P_{t}f(\mu)=  \ff 1\ll \E\bigg[f(X_{t}^{\mu})   \int_0^{t}  \Big\<\nn_{N_{s,t}^{\mu,\phi}}b_s(\cdot, P_s^*\mu)(X_s^\mu) + M_{s,t}^{\mu,\phi},\   \d W_s\Big\>  \bigg].  \end{equation}
 Consequently, $P_{t}f$ is intrinsically differentiable, and there exists a constant $c>0$ such that
\beq\label{ESN} \|D^IP_{t}f(\mu)\|_{L^{k^*}(\mu)}\le  \ff c {t^{2l-\frac{1}{2}}} \big(P_{t}|f|^{k^\ast}(\mu)\big)^{\ff 1 {k^\ast}} ,\ \ t\in (0,T], f\in\B_b(\R^{m+d}).\end{equation}
\end{thm}

\begin{proof}  It is easy to see that under {\bf (D)}, \eqref{ESN} follows from \eqref{BS3}.  So, it suffices to prove \eqref{BS3}.

Let $  X_t^\mu$ solve \eqref{E00} with initial value $ X_0^\mu$, and for any $\vv\in (0,1]$, let
$ Y_t^\vv$ solve \eqref{A2} for
 $Y_0=Y_0^\vv:=X_0^\mu+\vv\phi(X_0^\mu)$ and $  \nu=\mu^\vv.$
Then
 $$\L_{Y_0}=\L_{Y_0^\vv}=\mu^\vv.$$
 Let $\aa_{t_0}^\vv(s)$ be defined   in \eqref{A1} for $\nu=\mu^\vv$. By \eqref{a2} and \eqref{TTX}, we have
\beq\label{XXZ} \lim_{\vv\downarrow 0} \ff 1 \vv \aa_{t_0}^\vv (s)=  \aa_{t_0}^{\mu,\phi}(s),\ \ \ s\in [0,t_0],\end{equation}
while   \eqref{A3} and \eqref{ETS}  reduces to
\begin{equation*}\beg{split}  &  (Y_t^\vv)^{(2)}-  (X_t^\mu)^{(2)} =   \aa_{t_0}^\vv (t)+\vv\phi^{(2)}(X_0^\mu)  +\xi_t^{\mu^\vv}  -\xi_t^\mu,\\
& (Y_t^\vv)^{(1)}- (X_t^\mu)^{(1)} =  \vv \e^{tA} \phi^{(1)}(X_0^\mu)  +   \int_0^t \e^{(t-s)A} M \big\{ \aa_{t_0}^\vv(s)+\vv\phi^{(2)}(X_0^\mu)  +\xi_s^{\mu^\vv}-\xi_s^\mu\big\} \d s,\end{split}\end{equation*}
and
\begin{equation*}  \eta_t^\vv =\ff 1 \ll \Big\{b_t(Y_t^\vv, P_t^*\mu^\vv)- b_t(X_t^\mu, P_t^*\mu)- \{\aa_{t_0}^{\vv}\}'(t)\Big\},\ \ t\in[0,t_0].\end{equation*}
Then  by \eqref{a2} and \eqref{XXZ}, we have
\beq\label{A5'} \lim_{\vv\downarrow 0} \ff 1 \vv(Y_t^\vv-X_t^\mu)= N_{t,t_0}^{\mu,\phi}.  \end{equation}
 Let
$$R^\vv:= \e^{\int_0^{t_0} \<\eta_t^\vv,\d W_t\> -\ff 1 2 \int_0^{t_0} |\eta_t^\vv|^2\d t}.$$
By \eqref{NB'}, we obtain
$$ P_{t_0}^{\tt W,0} f(X_0^{\mu^\vv}):=  \E^{\tt W,0}[f(X_{t_0}^{\mu^\vv})]= \E^{\tt W,0}[R^\vv f(X_{t_0}^\mu)],\ \ f\in \B_b(\R^{m+d}).$$
As in \eqref{*W2},  by {\bf (D)}, \eqref{A5'} and \eqref{a2}, we derive
\beq\label{KM} \beg{split}& \lim_{\vv\downarrow0} \ff{P_{t_0}^{\tt W,0}f(X_0^{\mu^\vv}) - P_{t_0}^{\tt W,0} f(X_0^\mu)}\vv= \lim_{\vv\downarrow 0}
\E^{\tt W,0} \Big[f(X_{t_0}^\mu) \ff{R^\vv-1}\vv\Big]\\
&= \frac{1}{\lambda} \E^{\tt W,0}\bigg[ f(X_{t_0}^\mu) \int_0^{t_0}  \Big\<\nn_{N_{s,t_0}^{\mu,\phi}}b_s(\cdot, P_s^*\mu)(X_s^\mu) + M_{s,t_0}^{\mu,\phi},\   \d W_s\Big\>\bigg].\end{split}\end{equation}
Finally, similarly to the proof of  \eqref{*W}, since {\bf (D)} implies {\bf (C)} for $(\scr P_k(\R^{m+d}),\W_k)$ replacing $(\scr P_2(\R^{m+d}),\W_2)$,
the argument leading to  \eqref{YPP}  implies
$$P_{t_0}^{\tt W,0} \log f(X_{0}^\nu)- \log P_{t_0}^{\tt W,0} f(X_0^\mu)
 \le c(t_0)  \Big\{\W_k(\mu,\nu)^2  +    \rr_{t_0}(X_0^\mu,X_0^\nu)^2+  \sup_{t\in [0,t_0]} |\xi_t^\nu-\xi_t^\mu|^2\Big\} $$
 for some constant $c(t_0)>0$.
Therefore, as shown in  step (b) of the proof of Theorem \ref{TA2'},
  this enables us to apply the dominated convergence theorem  with \eqref{KM} to derive
 \beg{align*} &D_\phi^IP_{t_0}f(\mu) = \lim_{\vv\downarrow 0} \ff{\E[P_{t_0}^{\tt W,0} f(X_0^{\mu^\vv})- P_{t_0}^{\tt W,0} f(X_0^\mu)]}\vv=   \E\Big\{\lim_{\vv\downarrow 0}
\E^{\tt W,0} \Big[f(X_{t_0}^\mu) \ff{R^\vv-1}\vv\Big]\Big\} \\
&= \ff 1\ll \E\bigg[f(X_{t_0}^{\mu})   \int_0^{t_0}  \Big\<\nn_{N_{s,t_0}^{\mu,\phi}}b_s(\cdot, P_s^*\mu)(X_s^\mu) + M_{s,t_0}^{\mu,\phi},\   \d W_s\big\>  \bigg].  \end{align*}
 \end{proof}

\section{Exponential ergodicity in entropy}

Following the line of \eqref{RW}, we may use the log-Harnack inequality to study the exponential ergodicity in entropy. To this end, we consider the time homogeneous equation on $\R^d$
\beq\label{EES} \d X_t= b(X_t, \L_{X_t})\d t+\ll\d W_t+  \tt\si(\L_{X_t})\d \tt W_t, \ \ t\ge 0,\end{equation}
and the degenerate model on $\R^{m+d}$
\beq\label{EES2}
\begin{cases}
\d X_t^{(1)}=\big\{AX^{(1)}_t+MX_t^{(2)}\big\}\d t, \\
\d X_t^{(2)}=b(X_t,\L_{X_t})\d t+\ll\d W_t+  \tt\si(\L_{X_t})\d \tt W_t,\ \ t\ge 0,
\end{cases}
\end{equation}where $\ll>0$ is a constant.

\subsection{Non-degenerate case}
 \begin{enumerate}
\item[\bf{(E)}] There exist  constants $K, \theta_1,\theta_2>0$ with $\theta:=\theta_2-\theta_1>0$,    such that for any $\mu,\nu\in \scr P_2(\R^d)$ and $x,y\in \R^d$,
\begin{equation*}\begin{split}
 & |b(x,\mu)-b(y,\nu)|+|\tt \sigma (\mu)-\tt \sigma (\nu)|\leq K (|x-y|+\W_2(\mu,\nu)),\\
 & 2\<b(x,\mu)-b(y,\nu),x-y\>+ \|\sigma(\mu)-\sigma(\nu)\|^2_{HS}
 \leq -\theta_2|x-y|^2+\theta_1 \W_2(\mu,\nu)^2,
\end{split}\end{equation*}where $\|\cdot\|_{HS}$ is the Hilbert-Schmidt norm.
\end{enumerate}
By \cite[Theorem 2.1]{FYW1}, this assumption implies that \eqref{EES} is well-posed for distributions in $\scr P_2$, and $P_t^*$ has a unique invariant probability measure   $\bar\mu\in \scr P_2(\R^d)$ such that
\beq\label{EXP} \W_2(P_t^*\mu,\bar\mu)^2\le \e^{-\theta t} \W_2(\mu,\bar\mu)^2,\ \ t\ge 0.\end{equation}
The following result ensures the exponential convergence in entropy.

\begin{thm}\label{EXPEN}
Assume {\bf (E)} and let $P_t^*$ be associated with $\eqref{EES}$. Then there exists a constant $c>0$ such that
 \begin{align*}&\max\big\{\W_2(P_t^\ast\mu, \bar\mu)^2,\mathrm{Ent}(P_t^\ast\mu|\bar\mu)\big\}\\
&\leq c\e^{-\theta t}\min\big\{\W_2(\mu, \bar\mu)^2,\mathrm{Ent}(\mu|\bar\mu)\big\},\ \ \mu\in\scr P_2(\R^d),t\ge 1.\end{align*}
 \end{thm}
\begin{proof} According to the proof of \cite[Theorem 2.3]{RW}, {\bf (E)} implies   the Talagrand inequality
$$\W_2(\mu, \bar \mu)^2\leq c_1\mathrm{Ent}(\mu|\bar\mu), \ \ \mu\in\scr P_2(\R^d) $$ for some constant $c_1>0$.
According to   \cite[Theorem 2.1]{RW}, this together with  \eqref{EXP} and Theorem \ref{Loh} implies the desired assertion.

\end{proof}

\subsection{Degenerate case}

To study the exponential ergodicity for the degenerate model \eqref{EES2}, we extend the assumption $(A1)$-$(A3)$ in \cite{W17} to the present distribution dependent case.

\beg{enumerate} \item[\bf{(F)}]  $\tt\si$ and $b$ are Lipschitz continuous on $\scr P_2(\R^{m+d})$ and $\R^{m+d} \times \scr P_2(\R^{m+d})$ respectively. $(A,M)$ satisfies the rank condition \eqref{RRS} for some $1\le l\le m$,   and there exist constants $r>0, \theta_2>\theta_1\ge 0$ and $r_0\in (-\|M\|^{-1}, \|M\|^{-1})$ such that
\beg{align*} &\ff 1 2 \|\tt\si(\mu)-\tt\si(\nu)\|_{HS}^2 + \big\<b(x,\mu)-b(y,\nu),\ x^{(2)}-y^{(2)}+rr_0M^*(x^{(1)}-y^{(1)})\big\>\\
& +\big\<r^2(x^{(1)}-y^{(1)})+ rr_0M(x^{(2)}-y^{(2)}),\ A(x^{(1)}-y^{(1)})+ M(x^{(2)}-y^{(2)})\big\>\\
&\le \theta_1\W_2(\mu,\nu)^2-\theta_2|x-y|^2,\ \ x,y\in \R^{m+d},\ \mu,\nu\in \scr P_2(\R^{m+d}).\end{align*}\end{enumerate}
In the distribution free case, some examples are presented in \cite[Section 5]{W17}, which can be extended to the present setting if the Lipschitz constant of $\tt\si(\mu)$ and $b(x,\mu)$
in $\mu\in \scr P_2(\R^{m+d})$ is small enough.

\beg{thm} Assume {\bf (F)}. Then $P_t^*$ associated with $\eqref{EES2}$ has a unique  invariant probability measure $\bar\mu$, and there exist constants $c,\ll>0$ such that
$$\max\big\{\Ent(P_t^*\mu|\bar\mu), \W_2(P_t^*\mu,\bar\mu)^2\big\}\le c\e^{-\ll t} \W_2(\mu,\bar\mu)^2,\ \ t\ge 1, \mu\in \scr P_2(\R^{m+d}).$$
\end{thm}
\beg{proof} Let
$$\rr(x):= \ff{r^2} 2 |x^{(1)}|^2 +\ff 1 2 |x^{(2)}|^2 + rr_0 \<x^{(1)}, Mx^{(2)}\>,\ \ x=(x^{(1)},x^{(2)})\in \R^{m+d}.$$
By $r_0\|M\|<1$ and $r>0$, we find a constant $c_0\in (0,1) $ such that
\beq\label{N0} c_0   |x|^2 \le \rr(x)\le c_0^{-1} |x|^2,\ \ x\in\R^{m+d}.\end{equation}
Let $X_t$ and $Y_t$ solve \eqref{EES2} with initial values
\beq\label{N1} \L_{X_0}= \mu,\ \ \L_{Y_0}=\nu,\ \ \W_2(\mu,\nu)^2=\E[|X_0-Y_0|^2].\end{equation}
By {\bf (F)} and It\^o's formula, we obtain
\begin{align}\label{rhs}\d \rr (X_t-Y_t)\le \Big\{\theta_1 \W_2(P_t^*\mu,P_t^*\nu)^2 -\theta_2|X_t-Y_t|^2\Big\}\d t +\d M_t
\end{align}
for some martingale $M_t$,
 and
$$\d \rr (X_t)\le \Big\{\theta_1 \E[|X_t|^2] -\theta_2|X_t|^2+C+C|X_t|\Big\}\d t +\d \tilde{M}_t$$ for some martingale $\tilde{M}_t$ and constant  $C>0.$ In particular,
by \eqref{N0}, the latter implies
\beq\label{N*} \sup_{t\ge 0}   \E[|X_t|^2] <\infty.\end{equation}
 Since
\beq\label{N3}   \W_2(P_t^*\mu,P_t^*\nu)^2\le \E[|X_t-Y_t|^2],\end{equation} \eqref{rhs} and \eqref{N0} imply
$$\E[\rr(X_t-Y_t)]- \E[\rr(X_s-Y_s)]\le -c_0(\theta_2-\theta_1) \int_s^t \E[\rr(X_r-Y_r)]\d r,\ \ t\ge s\ge 0.$$
By Gronwall''s inequality, we derive
$$ \E[\rr(X_t-Y_t)]\le \e^{-c_0(\theta_2-\theta_1)t} \E[\rr(X_0-Y_0)],\ \ t\ge 0.$$
This together with \eqref{N0},  \eqref{N1} and \eqref{N3} yields
 \beg{align*} &\W_2(P_t^*\mu,P_t^*\nu)^2\le \E[|X_t-Y_t|^2] \le c_0^{-1}  \E[\rr(X_t-Y_t)]\le c_0^{-1}  \e^{-c_0(\theta_2-\theta_1)t} \E[\rr(X_0-Y_0)]\\
&\le c_0^{-2}  \e^{-c_0(\theta_2-\theta_1)t} \E[|X_0-Y_0|^2]= c_0^{-2}  \e^{-c_0(\theta_2-\theta_1)t}\W_2(\mu,\nu)^2,\ \ t\ge 0,\mu,\nu\in \scr P_2(\R^{m+d}).\end{align*}
As shown in \cite[Proof of Theorem 3.1(2)]{FYW1}, this together with \eqref{N*} implies that    $P_t^*$ has a unique invariant probability measure   $\bar\mu\in\scr P_2(\R^d)$,
and
\beq\label{N4} \W_2(P_t^*\mu,\bar\mu)^2\le c_0^{-2}  \e^{-c_0(\theta_2-\theta_1)t}\W_2(\mu,\bar\mu)^2,\ \ t\ge 0,\mu  \in \scr P_2(\R^{m+d}). \end{equation}

Finally,
by the log-Harnack inequality  \eqref{EC2}, there exists a constant $c_1>0$ such that
$$\Ent(P_1^*\mu|\bar \mu)\le c_1 \W_2(\mu,\bar \mu)^2.$$
Combining this with \eqref{N4} and using the semigroup property $P_{t}^*= P_{t-1}^*P_1^*$ for $t\ge 1$, we finish the proof.
\end{proof}

When $b$ is of a gradient type (induced by $\si$) as in \cite[(2.21)]{RW} such that the invariant probability measure $\bar\mu$ is explicitly given and satisfies
the Talagrand inequality, we may also derive the stronger upper bound  as in Theorem \ref{EXPEN}. We skip the details.

 \beg{thebibliography}{99}

\bibitem{ATW} M. Arnaudon, A. Thalmaier, F.-Y. Wang,  \emph{Gradient estimates and Harnack inequalities on non-compact Riemannian manifolds, } Stochastic Process. Appl. 119(2009), 3653-3670.
\bibitem{BH} Y. Bai, X. Huang, \emph{Log-Harnack inequality and exponential ergodicity for distribution dependent CKLS and Vasicek Model,} J. Theoret. Probab. (2022). https://doi.org/10.1007/s10959-022-01210-z.
\bibitem{B} D. Ba\~{n}os, \emph{The Bismut-Elworthy-Li formula for mean-field stochastic differential equations,} Ann. Inst. Henri Poincar\'{e} Probab. Stat. 54 (2018) 220-233.
 \bibitem{BRW} J. Bao, P. Ren, F.-Y. Wang, \emph{Bismut formula for Lions derivative of distribution-path dependent SDEs,} J. Differential Equations 282(2021), 285-329.
\bibitem{BBP} M. Bauer, T. M-Brandis, F. Proske, \emph{Strong solutions of mean-field stochastic differential equations with irregular drift,} Electron. J. Probab. 23(2018), 1-35.
 \bibitem{Bim} J. M. Bismut, \emph{Large Deviations and the Malliavin Calculus,} Boston: Birkh\"{a}user, MA, 1984.
\bibitem{BRS} V. I. Bogachev, M. R\"{o}ckner, S. V. Shaposhnikov,
\emph{Distances between transition probabilities of diffusions and applications to nonlinear Fokker-Planck-Kolmogorov equations,} J. Funct. Anal. 271 (2016), 1262-1300.
\bibitem{CM} D. Crisan, E. McMurray, \emph{Smoothing properties of McKean-Vlasov SDEs,}Probab. Theory Relat. Fields 171(2018), 97-148.
\bibitem{EL}  K. D.   Elworthy, X.-M. Li, \emph{Formulae for the derivatives of heat semigroups,} J. Funct. Anal. 125(1994), 252-286.
\bibitem{GW}A. Guillin, F.-Y. Wang, \emph{Degenerate Fokker-Planck equations: Bismut formula, gradient estimate and Harnack inequality,} J. Differential Equations 253 (2012) 20-40.
\bibitem{HRW} X. Huang, P. Ren, F.-Y. Wang, \emph{Distribution dependent stochastic differential equations, } Front. Math. China 16(2021), 257-301.
\bibitem{HRW19} X. Huang, M. R\"{o}ckner, F.-Y. Wang, \emph{Non-linear Fokker--Planck equations for probability measures on path space and path-distribution dependent SDEs,}   Discrete Contin. Dyn. Syst. 39(2019), 3017-3035.
\bibitem{HS} X. Huang, Y. Song, \emph{Well-posedness and regularity for distribution dependent SPDEs with singular drifts,} Nonlinear Anal. 203(2021), 112167.



\bibitem{HW18} X. Huang, F.-Y. Wang, \emph{Distribution dependent SDEs with singular coefficients,} Stochastic Process. Appl. 129(2019), 4747-4770.


\bibitem{HW22a} X. Huang, F.-Y. Wang, \emph{Log-Harnack inequality and Bismut formula for singular McKean-Vlasov SDEs,} arXiv:2207.11536.
\bibitem{HW21} X. Huang, F.-Y. Wang, \emph{Derivative estimates on distributions of McKean-Vlasov SDEs,} Electron. J. Probab. 26(2021), 1-12.






\bibitem{PW} E. Priola, F.-Y. Wang, \emph{Gradient estimates for diffusion semigroups with singular coefficients,} J. Funct. Anal. 236(2006), 244-264.

\bibitem{RW19} P. Ren, F.-Y. Wang, \emph{Bismut formula for Lions derivative of distribution dependent SDEs and applications,} J. Differential Equations 267(2019), 4745-4777.
\bibitem{RW}  P. Ren, F.-Y. Wang, \emph{Exponential convergence in entropy and Wasserstein for McKean-Vlasov SDEs,}  Nonlinear Anal. 206(2021), 112259.
\bibitem{RW0} P. Ren, F.-Y. Wang, \emph{Derivative formulas in measure on Riemannian manifolds,} Bull. Lond. Math. Soc. 53(2021), 1786-1800.
\bibitem{S} T. Seidman,\emph{ How violent are fast controls?}  Math. Control Signals Systems  1(1988), 89-95.


\bibitem{Song} Y. Song, \emph{Gradient estimates and exponential ergodicity for mean-field SDEs with jumps,} J. Theoret. Probab. 33(2020),201-238.

\bibitem{T} M. Tahmasebi, \emph{The Bismut-Elworthy-Li formula for semi-linear distribution-dependent SDEs driven by fractional Brownian motion,} arXiv:2209.05586.
\bibitem{W97} F.-Y. Wang,    \emph{Logarithmic Sobolev inequalities on noncompact Riemannian manifolds,} Probab. Theory Related Fields 109(1997),417-424.
\bibitem{W10} F.-Y. Wang,  \emph{Harnack inequalities on manifolds with boundary and applications,} J. Math. Pures Appl. 94(2010), 304-321.
\bibitem{Wbook} F.-Y. Wang, \emph{Harnack Inequality for Stochastic Partial Differential Equations,} Springer, New York, 2013.
\bibitem{W17} F.-Y. Wang, \emph{Hypercontractivity and applications for stochastic Hamiltonian systems, } J. Funct. Anal. 272(2017), 5360-5383.
\bibitem{FYW1} F.-Y. Wang, \emph{Distribution-dependent SDEs for Landau type equations,} Stochastic Process. Appl. 128(2018), 595-621.


\bibitem{FYW3} F.-Y. Wang, \emph{Derivative Formula for Singular McKean-Vlasov SDEs,} Commun. Pure Appl. Anal. 22(2023),  1866-1898.
\bibitem{WZ1} F.-Y. Wang, X. Zhang, \emph{Derivative formula and applications for degenerate diffusion semigroups,} J. Math. Pures Appl. 99(2013),726-740.








\end{thebibliography}

\end{document}